\documentclass{amsart}

\usepackage[all]{xy}
\usepackage{amsmath}
\usepackage{hyperref}
\hypersetup{
    colorlinks=true,    
    linkcolor=blue,          
    citecolor=blue,      
    filecolor=blue,      
    urlcolor=blue           
}
\usepackage{tikz}
\usetikzlibrary{graphs,positioning,arrows,shapes.misc,decorations.pathmorphing}

\tikzset{
    >=stealth,
    every picture/.style={thick},
    graphs/every graph/.style={empty nodes},
}

\tikzstyle{vertex}=[
    draw,
    circle,
    fill=black,
    inner sep=1pt,
    minimum width=5pt,
]
\usepackage[position=top]{subfig}
\usepackage[alphabetic,backrefs]{amsrefs}
\usepackage{amssymb}
\usepackage{color}

\calclayout

\usetikzlibrary{decorations.pathmorphing}
\tikzstyle{printersafe}=[decoration={snake,amplitude=0pt}]

\newcommand{\vol}{\operatorname{vol}}

\renewcommand{\qq}{\mathbb{Q}}

\newtheorem{introthm}{Theorem}

\newtheorem{introconj}{Conjecture}
\newtheorem{introcor}{Corollary}

\newtheorem{theorem}{Theorem}[section]
\newtheorem{lemma}[theorem]{Lemma}
\newtheorem{proposition}[theorem]{Proposition}

\newtheorem{definition}[theorem]{Definition}
\newtheorem{example}[theorem]{Example}

\newtheorem{remark}[theorem]{Remark}

\theoremstyle{remark}

\numberwithin{equation}{section}

\begin{document}

\title[On minimal log discrepancies and Koll\'ar components]{On minimal log discrepancies and Koll\'ar components}

\author[J.~Moraga]{Joaqu\'in Moraga}
\address{
Department of Mathematics, University of Utah, 155 S 1400 E, JWB 321,
Salt Lake City, UT 84112, USA}
\email{moraga@math.utah.edu}

\subjclass[2010]{Primary 14E30, 
Secondary 14B05.}
\maketitle

\begin{abstract}
In this article we prove a local implication of boundedness of Fano varieties.
More precisely, we prove that $d$-dimensional $a$-log canonical singularities, with standard coefficients, 
which admit an $\epsilon$-plt blow-up have minimal log discrepancies belonging to a finite set which only depends on $d,a$ and $\epsilon$.
This result gives a natural geometric stratification of the possible mld's in a fixed dimension by finite sets.
As an application, we prove the ascending chain condition for minimal log discrepancies of exceptional singularities. 
We also introduce an invariant for klt singularities related to the total discrepancy of Koll\'ar components.
\end{abstract}

\tableofcontents

\section{Introduction}

The development of projective birational geometry has been deeply connected with the understanding of singularities~\cite{Xu17}.
In particular, global theories shed light on the singularities of the minimal model program~\cite{Kol13}.
However, singularities of dimension greater than three seem too complicated to have an explicit characterization~\cite{Kol11}.
Therefore, a more qualitative and intrinsic description of these singularities is desirable.

A common technique to study singularities using birational geometry is to apply certain monoidal transformation to the singularity
to extract an exceptional projective divisor over it, and then try to deduce some local information of the singularity
from the global information of the exceptional divisor. 
This approach has been successful in many cases: The study of dual complexes of singularities~\cite{KX16,dFKX17},
the finiteness of the algebraic fundamental group of a klt singularity~\cite{Xu14},
the ascending chain condition for log canonical thresholds~\cite{dFEM10,HMX14},
the study of the normalized volume function on klt singularities~\cite{Li17,LX16,LX17},
and the theory of log canonical complements~\cite{PS01,PS09,Bir16a}, among others.

In this article, we use this approach to study the minimal log discrepancies of klt singularities~\cite{Amb99,Sho04}.
We aim to explain how a bound on the singularities of a plt blow-up implies finiteness of the possible minimal log discrepancies.
More precisely, we say that a plt blow-up $\pi \colon Y\rightarrow X$ at a point $x$ of a klt pair $(X,\Delta)$ is $\epsilon$-plt
if the log discrepancies of the corresponding plt pair $(Y,\Delta_Y+E)$ are either zero or greater than $\epsilon$ (see Definition~\ref{def:e-plt blow-up}).
We say that the point $x$ of the klt pair $(X,\Delta)$ admits an $\epsilon$-plt blow-up if there exists $\pi$ as above.
For $\epsilon$ a positive real number, adjunction~\cite[Theorem 0.1]{Hac14} and the boundedness of Fano varieties~\cite[Theorem 1.1]{Bir16b}, allows us to conclude that the normal projective varieties $E$ belong to a bounded family.
By~\cite[Lemma 1]{Xu14}, we know that every klt singularity admits a plt blow-up and a simple argument using resolution of singularities 
proves that every klt singularity admits an $\epsilon$-plt blow-up for some positive real number $\epsilon$ (see Lemma~\ref{e-plt blow-up}).

Using the above notation we can introduce the set of minimal log discrepancies 
of $a$-log canonical pairs admitting an $\epsilon$-plt blow-up:
\[
\mathcal{M}(d,\mathcal{R})_{a,\epsilon}:=
\left\{ 
{\rm mld}_x(X,\Delta)
\,\middle\vert\, 
\begin{tabular}{c}   
$\dim(X)=d, {\rm coeff}(\Delta) \in \mathcal{R},$ and $(X,\Delta)$ is a $a$-lc pair \\
 which admits an $\epsilon$-plt blow-up at $x$
\end{tabular}
\right\}.
\]
We recall the conjecture known as the ascending condition for minimal log discrepancies:

\begin{introconj}\label{accmld}
Let $d$ be a positive integer and $\mathcal{R}$ be a set of real numbers satisfying the descending chain condition.
Then the set $\mathcal{M}(d,\mathcal{R})_{0,0}$ satisfies the ascending chain condition.
\end{introconj}

The ascending chain condition for minimal log discrepancies is known 
for surface singularities~\cite{Ale93}, for certain terminal threefold singularities~\cite[Lemma 4.4.1]{Sho96}, and for toric singularities~\cite{Bor97,Amb06}.
However, in higher dimensions there is not much that we can say about the possible mld's in a fixed dimension.
In this paper, we give a first step towards the understanding of higher dimensional minimal log discrepancies.
The following theorem can be understood as a natural geometric stratification of the possible mld's of a fixed dimension by finite sets:

\begin{introthm}\label{finiteness e-plt blow-up}
Let $d$ be a positive integer and let $a$ and $\delta$ be positive real numbers, and $\mathcal{S}$ the set of standard rational numbers.
Then the set $\mathcal{M}(d,\mathcal{S})_{a,\epsilon}$ is finite.
\end{introthm}

The theorem implies the following corollary towards the ascending chain condition.

\begin{introcor}\label{acc and e-plt blow-ups}
The set $\mathcal{M}(d,\mathcal{S})_{0,\epsilon}$ satisfies the ascending chain condition.
\end{introcor}

In Section~\ref{sec: examples} we give two examples that show that the theorem does not hold if $a$ and $\epsilon$ are not positive.
We also prove that $a$-log canonical singularities which admit an $\epsilon$-plt blow-up have bounded Cartier index.

\begin{introthm}\label{bounded cartier index e-plt}
Let $d$ be a positive integer, and let $a$ and $\epsilon$ be positive real numbers.
There exists $p$ only depending on $d,a$ and $\epsilon$ satisfying the following.
Let $(X,\Delta)$ be a $d$-dimensional $a$-log canonical pair with standard coefficients which admits an $\epsilon$-plt blow-up at $x\in X$.
Then $p(K_X+\Delta)$ is a Cartier divisor on a neighborhood of $x\in X$.
\end{introthm}

Exceptional singularities are those klt singularities for which any log canonical threshold is computed at a unique divisorial valuation~\cite[Definition 1.5]{Sho00}.
The exceptional Du Val surface singularities are the $E_6,E_7$ and $E_8$ singularities.
Hypersurfaces exceptional singularities were studied by Ishii and Prokhorov~\cite{IP01}.
In dimension $3$, Prokhorov and Markusevich proved that there are only finitely many $\epsilon$-log canonical exceptional quotient singularities~\cite{MP99}.
However, a classification of exceptional singularities in higher dimensions seems unfeasable.
In this direction, we prove the following application of our main theorem:

\begin{introcor}\label{acc exceptional singularities}
The ascending chain condition for minimal log discrepancies of exceptional singularities with standard coefficients holds.
\end{introcor}

It is worth mentioning that Corollary~\ref{acc exceptional singularities} follows almost directly from the proof of~\cite[Theorem 4.4]{PS01} and~\cite[Theorem 1.1]{Bir16b}. These two results together with~\cite{Ish00} and~\cite{Fuj01} are the main motivation of Theorem~\ref{finiteness e-plt blow-up}. The proof of~\cite[Theorem 4.4]{PS01} already contains some of the ideas used in this article.

\subsection*{Acknowledgements} 
The author would like to thank Caucher Birkar, Christopher Hacon, Harold Blum, Stefano Filipazzi, and Yuri Prokhorov for many useful comments. 
The author was partially supported by NSF research grants no: DMS-1300750, DMS-1265285 and by a
grant from the Simons Foundation; Award Number: 256202.
After completing this project, the author was informed that J. Han, J. Liu and V. Shokurov have obtained 
the results of this paper with more general coefficients~\cite{HLS18}.

\section{Preliminaries}
\label{sec:2}

All varieties in this paper are quasi-projective over a fixed algebraically closed field of characeristic zero unless stated otherwise.
In this section we collect some definitions and preliminary results which will be used in the proof of the main theorem.

\subsection{Singularities}

In this subsection, we recall the singularities of the minimal model program, the set of standard coefficients and exceptional singularities.
We also prove some basic properties about singularities.

\begin{definition}{\em  
 In this paper a {\em sub-pair} $(X,\Delta)$ consists of a normal quasi-projective variety $X$ and a $\qq$-divisor $\Delta$ so that $K_X+\Delta$ is a $\qq$-Cartier $\qq$-divisor. 
If the coefficients of $\Delta$ are non-negative then we say that $(X,\Delta)$ is a {\em log pair}, or simply a {\em pair}. 

Let $\pi \colon W \rightarrow X$ be a log resolution of the pair $(X,\Delta)$ and denote by $K_W+\Delta_W+F_W$ the log pull-back of $K_X+\Delta$,
where $\Delta_W$ is the strict transform of $\Delta$ on $W$ and $F_W$ is an exceptional divisor.
The {\em discrepancy} of a prime divisor $E$ on $W$ with respect to the pair $(X,\Delta)$ is 
\[
a_E( X,\Delta) :=  -{\rm coeff}_E(\Delta_W+F_W).
\]
The {\em log discrepancy} of a prime divisor $E$ on $W$ with respect to the pair $(X,\Delta)$ is the value
\[
a_E(X,\Delta)+1.
\]
The {\em center} of $E$ on $X$ is its image on $X$ via the morphism $\pi$.
We denote by $c_X(E)$ the center of the prime divisor $E$ on the variety $X$.
We say that the sub-pair $(X,\Delta)$ is {\em sub-$\epsilon$-log canonical} 
if 
\[
a_E(X,\Delta)\geq  -1 +\epsilon,
\]
for every prime divisor $E$ on $W$. 
If $(X,\Delta)$ is a pair then we say that is $\epsilon$-log canonical in the above situation.
If $\epsilon>0$ is arbitrary we may also say that $(X,\Delta)$ is {\em Kawamata log terminal} (or klt)
and if $\epsilon=0$ we just say that the pair is {\em log canonical}, equivalently, the center of a log canonical place.
The {\em total discrepancy} $a(X,\Delta)$ of the pair $(X,\Delta)$ is the infimum between all discrepancies 
$a_E(X,\Delta)$ with $E$ a prime divisor over $X$.
Thus, $(X,\Delta)$ is $a$-log canonical if and only if $a(X,\Delta)+1\geq a$.

Let $(X,\Delta)$ be a log canonical pair. A {\em log canonical place} of $(X,\Delta)$ is a prime divisor $E$ on
a birational model of $X$ so that $a_E(X,\Delta)=-1$.
A {\em log canonical center} of $(X,\Delta)$ is the image on $X$ of a log canonical place.
}
\end{definition}
	
\begin{definition}
{\em 
Let $(X,\Delta)$ be a log pair and $x\in X$. 
The {\em minimal log discrepancy} of $(X,\Delta)$ at $x$ is
\[
{\rm mld}_x(X,\Delta):={\rm inf} \left\{ 
a_E(X,\Delta)+1 \mid \text{ $E$ is a prime divisor over $X$ so that $c_X(E)=x$}
\right\}.
\]
If $(X,\Delta)$ is a log canonical pair, taking a log resolution and using~\cite[Lemma 2.29]{KM98} 
we can see that the above infimum is indeed a minimum.
Observe that ${\rm mld}_x(X,\Delta)\geq 0$ if and only if the pair $(X,\Delta)$ is log canonical at $x\in X$.
Moreover, ${\rm mld}_x(X,\Delta)>0$ if and only if the pair $(X,\Delta)$ is klt at $x\in X$.
On the other hand, if $(X,\Delta)$ is not log canonical at $x\in X$, then ${\rm mld}_x(X,\Delta)=-\infty$ (see, e.g.~\cite[Corollary 2.32]{KM98}).
}
\end{definition}

\begin{definition}{\em 
We say that the pair $(X,\Delta)$ is {\em divisorial log terminal} (or dlt) if the following conditions hold:
\begin{enumerate} 
\item there exists a closed subset $Z\subset X$ so that $X\setminus Z$ is smooth,
\item $\Delta|_{X\setminus Z}$ has simple normal crossing support, and 
\item every divisor $E$ over $X$ with center on $Z$ has positive log discrepancy with respect to $(X,\Delta)$.
\end{enumerate}
A pair $(X,\Delta)$ is called {\em purely log terminal} (or plt) if the log discrepancy of every exceptional prime divisor
over $X$ is strictly positive. In this case $\lfloor \Delta \rfloor$ is a disjoint union of normal prime divisors.}
\end{definition}

\begin{definition}\label{def:standard coeff}
{\em 
Given a finite set $\mathcal{R}$ of rational numbers, we denote by 
\[
\mathcal{H}(\mathcal{R}):= \{1\} \cup \left \{ 1 -\frac{r}{m} \,\middle\vert\,  r\in \mathcal{R} \text{ and } m\in \mathbb{Z}_{>0} \right\},
\]
and call $\mathcal{H}(\mathcal{R})$ the {\em set of hyperstandard coefficients} associated to $\mathcal{R}$.
We denote $\mathcal{S}:=\mathcal{H}(\{1\})$ and call this the set of {\em standard coefficients}.
In what follows, we will focus on pairs $(X,\Delta)$ so that the coefficients of the boundary ${\rm coeff}(\Delta)$
belong to the set of standard coefficients. In this case, we say that the coefficients of $\Delta$ are {\em standard}.}
\end{definition}

\begin{definition}\label{def:e-plt blow-up}
{\em A {\em plt blow-up} of a log pair $(X,\Delta)$ at a point $x\in X$ is a projective birational morphism 
$\pi \colon Y \rightarrow X$ with the following properties:
\begin{enumerate}
\item $Y$ is a quasi-projective normal variety,
\item the exceptional locus of $\pi$ is an irreducible divisor $E$ whose image on $X$ is $x$, 
\item the pair $(Y,\Delta_Y+E)$ is purely log terminal, where $\Delta_Y$ is the strict transform of $\Delta$ on $Y$, and 
\item $-E$ is ample over $X$.
\end{enumerate}
We say that the pair $(X,\Delta)$ admits a plt blow-up at $x\in X$ if there exists $\pi$ with the above conditions.
Moreover, we say that the plt blow-up is an {\em $\epsilon$-plt blow-up} if any log discrepancy of $(Y,\Delta_Y+E)$ is either zero or greater than $\epsilon$.
Analogously, we say that the pair $(X,\Delta)$ admits an $\epsilon$-plt blow-up at $x\in X$ if there exists $\pi$ with the above conditions.
}
\end{definition}

\begin{definition}{\em 
The exceptional divisor of a plt blow-up are often called {\em Koll\'ar components} of the singularity~\cite[Definition 1.1]{LX17}.
We may use both acceptions on this paper.
We may also call a {\em Koll\'ar component} over the klt pair $(X,\Delta)$ the log Fano pair 
\[
K_E+\Delta_E := (K_Y+\Delta_Y+E)|_E
\]
obtained by adjunction to the $E\subset Y$. 
Observe that $-(K_E+\Delta_E)$ is ample and $(E,\Delta_E)$ is klt.
Hence, $(E,\Delta_E)$ is a log Fano pair~\cite[2.10]{Bir16a}.
}
\end{definition}

\begin{proposition}\label{e-plt blow-up}
Let $(X,\Delta)$ be a klt pair and $x\in X$. There exists an $\epsilon$-plt blow-up of $(X,\Delta)$ at $x$ for some positive $\epsilon$.
\end{proposition}

\begin{proof}
By~\cite[Lemma 1]{Xu14}, we can construct a plt blow-up $\pi \colon Y \rightarrow X$ of $(X,\Delta)$ over $x$.
Let $\pi_Y \colon W\rightarrow Y$ be a log resolution of $(Y,\Delta_Y+E)$, so we can write
\[
K_W+\Delta_W+F_W = \pi_Y^*(K_Y+\Delta_Y+E),
\]
where $\Delta_W$ is the strict transform of $\Delta_Y$ on $W$, and $F_W$ is an exceptional divisor.
We can take $\epsilon$ small enough so that 
\[
{\rm coeff}(\Delta_W + F_W - E_W ) < 1-\epsilon
\] 
where $E_W$ is the strict transform of $E$ on $W$.
By~\cite[Lemma 2.29]{KM98} we conclude that every exceptional divisor over $W$ has log discrepancy
either zero or greater than $\epsilon$ with respect to the purely log terminal pair $(Y,\Delta_Y+E)$.
\end{proof}

Theorem~\ref{finiteness e-plt blow-up} motivates the following natural invariant of klt singularities.

\begin{definition}\label{mKc}{\em
Let $(X,\Delta)$ be a klt pair and $x\in X$.
We define the {\em mildest component}, 
or $\mathcal{M}\mathcal{C}$ for simplicity, of $(X,\Delta)$ at the point $x\in X$ to be:
\[
{\rm \mathcal{M}\mathcal{C}}_x(X,\Delta) := 
\sup
\{
a(E,\Delta_E)+1 \mid \text{ $(E,\Delta_E)$ is a Koll\'ar component	 of $(X,\Delta)$ over $x\in X$}
\}.
\]
In Proposition~\ref{mildest Kollar component is a minimum}, we will prove that the $\mathcal{M}\mathcal{C}$ is indeed attained by some Koll\'ar component over $x\in X$.
In this setting, the conditions of Theorem~\ref{finiteness e-plt blow-up} and Theorem~\ref{bounded cartier index e-plt} can be abbreviated as 
$(X,\Delta)$ is a $d$-dimensional pair which is $a$-log canonical at $x\in X$ and ${\rm \mathcal{M}\mathcal{C}}_x(X,\Delta)\geq \epsilon$.
}
\end{definition}

\begin{remark}{\em 
In general one could define the mildest component of $(X,\Delta)$ at $x\in X$
for every Grothendieck point of the variety $X$.
However, by cutting down with general hyperplanes, 
the study of minimal log discrepancies on a variety $X$ of dimension $d$ at a point $x\in X$ of codimension $k$ 
is equivalent to the stuty of minimal log discrepancies on a variety $X$ of dimension $d-k$ at a closed point.}
\end{remark}

\begin{definition}{\em
A {\em dlt modification} of a log canonical pair $(X,\Delta)$ is a projective birational morphism $\pi \colon Y\rightarrow X$ so
that $\pi^*(K_X+\Delta)=K_Y+\Delta_Y+E_Y$, where $\Delta_Y$ is the strict transform of $\Delta$ on $X$ and $E_Y$ is an exceptional divisor over $X$
so that the sub-pair $(Y,\Delta_Y+E_Y)$ is a dlt pair.
It is known that every log canonical pair admits a dlt modification~\cite[Theorem 3.1]{KK13}.}
\end{definition}

\begin{definition}
{\em We say that a klt pair $(X,\Delta)$ is {\em exceptional} at $x\in X$ if for every boundary $\Gamma \geq 0$ on $X$
so that $(X,\Delta+\Gamma)$ is a log canonical pair, any dlt modification of $(X,\Delta+\Gamma)$ is indeed plt.}
\end{definition}

\begin{remark}{\em 
It is known that if $(X,\Delta)$ is exceptional at $x$ and $X$ is $\qq$-factorial, then there exists a unique prime divisor over $X$ which 
compute all log canonical thresholds over $x\in X$~\cite[Proposition 2.9]{MP99}.}
\end{remark}

\begin{lemma}\label{exceptional singularities}
Let $(X,\Delta)$ be a klt pair which is exceptional at $x\in X$.
Assume that the coefficients of $\Delta$ are standard.
There exists an $\epsilon_d$-blow-up of $(X,\Delta)$ at $x$, for some positive real number $\epsilon_d$ which only depends on the dimension $d$ of $X$.
\end{lemma}

\begin{proof}
By Lemma~\ref{e-plt blow-up}, we can extract an $\epsilon$-plt blow-up over $x\in X$ for some real number $\epsilon$.
By~\cite[Theorem 1.8]{Bir16a} there exists a strong $n$-complement $K_Y+\Delta_Y+\Gamma_Y+E$ for $K_Y+\Delta_Y+E$ over $x$,
for $n$ only depending on $d$ and $\mathcal{S}$.
Since $\mathcal{S}$ is a fixed set then $n$ only depends on $d$.
Since $K_Y+\Delta_Y+\Gamma_Y+E$ is $\qq$-linearly trivial over $x\in X$ then we can write
\[
K_Y+\Delta_Y+\Gamma_Y+E = \pi^*(K_X+\Delta+\Gamma)
\]
for some boundary $\Gamma$ on $X$ so that $(X,\Delta+\Gamma)$ is a log canonical pair.
By the exceptionality of $(X,\Delta)$ at $x$ we deduce that $(X,\Delta+\Gamma)$ has a unique log canonical place 
and hence $(Y,\Delta_Y+\Gamma_Y+E)$ has a unique log canonical place $E$.
Moreover, since $n(K_Y+\Delta_Y+\Gamma_Y+E)$ is Cartier over $x$, we conclude that 
\[
a_F( Y, \Delta_Y+E) +1 \geq a_F(Y, \Delta_Y+\Gamma_Y +E) + 1 \geq \frac{1}{n}
\]
for every divisor $F$ over $Y$ which is not equal to $E$.
Thus, it suffices to take $\epsilon_d=\frac{1}{n}$.
\end{proof}

\subsection{Bounded families}

In this subsection, we recall the definition of a log bounded family and prove some properties of such families.

\begin{definition}{\em 
We say that a set of pairs $\mathcal{P}$ is {\em log bounded} if there exists a projective morphism
$\mathcal{X} \rightarrow T$ of possibly reducible varieties and a divisor $\mathcal{B}$ on $\mathcal{X}$ so that for every pair
$(X,\Delta)\in \mathcal{P}$ there exists a closed point $t\in T$ and an isomorphism 
$\phi \colon  \mathcal{X}_t \rightarrow X$ so that $(\mathcal{X}_t,\mathcal{B}_t)$ is a pair and
$\phi_*^{-1}D\leq \mathcal{B}_t$.
If moreover $\phi$ induces an isomorphism of pairs
between $(X,\Delta)$ and $(\mathcal{X}_t, \mathcal{B}_t) $, meaning that 
for every prime divisor $D$ on $X$ we have that 
\[
{\rm coeff}_D( \Delta) = {\rm coeff}_{\phi^{-1}_*D}(\mathcal{B}_t),
\]
we will say that the set of pairs $\mathcal{P}$ is {\em strictly log bounded}.
We say that $\mathcal{X}\rightarrow T$ is a {\em bounding family} for the varieties $X$ 
and that $\mathcal{B}$ is a {\em bounding divisor} for the set of divisors $\{ \Delta \mid (X,\Delta)\in \mathcal{P}\}$.}
\end{definition}

The following lemma follows from the definition of strictly log bounded family.

\begin{lemma}\label{from lb to slb}
Let $\mathcal{P}$ be a log bounded family of pairs so that the set $\{ {\rm coeff}(\Delta) \mid (X,\Delta) \in \mathcal{P}\}$
is finite. Then the family $\mathcal{P}$ is strictly log bounded.
\end{lemma}

\begin{lemma}\label{from bounded to log bounded}
Let $\mathcal{P}$ be a bounded family of $d$-dimensional projective varieties 
and $\mathcal{Q}$ a set of pairs $\{ (X,\Delta) \mid X \in \mathcal{P}\}$,
so that ${\rm coeff}(\Delta)$ satisfies the descending chain condition
and for every $(X,\Delta)\in \mathcal{Q}$ we have that $-(K_X+\Delta)$ is pseudo-effective.
Then $\mathcal{Q}$ is a log bounded set of pairs.
\end{lemma}

\begin{proof}
We can find a positive constant $C$  so that for each $X\in \mathcal{P}$ there is a very ample Cartier divisor $A$ with $A^d\leq C$.
Morever, we can assume that $A^{d-1}(-K_X) \leq C$.
The set ${\rm coeff}(\Delta)$ satisfies the descending chain condition, so there exists $\delta>0$ small enough so that
$\delta < {\rm coeff}(\Delta)$ for every boundary $\Delta$ of a pair $(X,\Delta)$ on $\mathcal{Q}$.
Since $-(K_X+\Delta)$ is pseudo-effective we have that $A^{d-1}\cdot (-K_X+\Delta) \geq 0$, hence we conclude that
\[
A^{d-1}\cdot \left( \frac{1}{\delta} \Delta_{\rm red} \right)  \leq A^{d-1}\cdot \Delta \leq A^{d-1} \cdot \left( -K_X \right) \leq C,
\]
where $\Delta_{\rm red}$ is $\Delta$ with the reduced structure.
Thus, we get that $A^{d-1}\cdot \Delta_{\rm red}\leq \delta C$, so by~\cite[Lemma 3.7.(2)]{Ale94} we conclude that the set of pairs $\mathcal{Q}$ is log bounded.
\end{proof}

\begin{definition}{\em 
Let $X$ be an irreducible projective variety of dimension $d$ and let $D$ be a $\qq$-Cartier divisor on $X$.
The {\em volume} of $D$ is
\[
{\rm vol}(D):=\limsup_{m\rightarrow \infty} \frac{d! h^0(X,\mathcal{O}_X(mD))}{m^n}.
\]
In particular, a big divisor has positive volume.}
\end{definition}

\begin{lemma}\label{bounded volume}
Let $\mathcal{P}$ be a strictly log bounded family of $d$-dimensional klt pairs.
Then there exist positive real numbers $v_1$ and $v_2$, only depending on $\mathcal{P}$, 
so that for every $(X,\Delta) \in \mathcal{P}$ with $-(K_X+\Delta)$ ample we have that
\[
v_1 < \vol(-(K_X+\Delta)) < v_2.
\]
\end{lemma}

\begin{proof}
Since ampleness is open in families, we may restrict to an open set $U\subset T$ 
so that $-(K_{\mathcal{X}_t} +\mathcal{B}_t)$ is ample for every $t\in U$.
Furthermore, by Noetherian induction we may assume that the induced morphism $\mathcal{X}_U \rightarrow U$
is a projective smooth morphism of relative dimension $d$ with normal fibers.
In this case, we have that
\[
\vol( -(K_{\mathcal{X}_t}+\mathcal{B}_t)) = ( -(K_{\mathcal{X}_t}+\mathcal{B}_t))^d
\]
is an upper-semicontinuous function on $t\in U$.
Hence, it takes finitely many values on $U$.
\end{proof}

\begin{lemma}\label{log bounded curve}
Let $\mathcal{P}$ be a strictly log bounded family of $d$-dimensional klt pairs.
We can find a positive real number $M$, only depending on $\mathcal{P}$,
so that for every $(X,\Delta)\in \mathcal{P}$ with $-(K_X+\Delta)$ ample
there exists an ample curve $C$ on $X$ so that 
\[
-(K_X+\Delta)\cdot C \leq M.
\]
\end{lemma}

\begin{proof}
Since the family of pairs is strictly log bounded, we may find a positive natural number $m$, only depending on $\mathcal{P}$, 
so that $|-m(K_X+\Delta)|$ is a base point free linear system for any $(X,\Delta)$ as in the statement.
Hence, a general curve $C$ on the rational equivalence class of $(-m(K_X+\Delta))^{d-1}$ will be
an ample curve. Moreover, by Lemma~\ref{bounded volume}, we have that 
\[
-(K_X+\Delta) \cdot C = m^{d-1} (-(K_X+\Delta))^{d} = m^{d-1} \vol(-(K_X+\Delta)) < m^{d-1}v_2.
\]
Thus, it suffices to take $M=m^{d-1}v_2$.
\end{proof}

\begin{lemma}\label{bound of ld for Gorenstein}
Let $d$ be a positive integer and $\epsilon$ a positive real number.
There exists a constant $l$ only depending on $d$ and $\epsilon$ satisfying the following.
Let $X$ be a $d$-dimensional klt variety so that $K_X$ is Cartier at $x\in X$.
Assume that $X$ admits an $\epsilon$-plt blow-up at $x\in X$ extracting a divisor $E$. 
Then we have that
\[
a_E(X,0) \leq l.
\]
In particular, there are finitely many possible values for $a_E(X,0)$.
\end{lemma}

\begin{proof}
Let $\pi\colon Y \rightarrow X$ be the $\epsilon$-plt blow-up at $x\in X$ and write 
\[
K_Y-a_E(X,0)E = \pi^*K_X.
\]
By assumption of $X$ being klt and $K_X$ being Cartier at $x$ we know that $a_E(X,0)$ is a non-negative integer.
We write
\[
K_E+\Delta_E = ( K_Y+E)|_E.
\]
Thus, the pair $(E,\Delta_E)$ is $\epsilon$-lc and $-(K_E+\Delta_E)$ is ample.
By boundedness of Fano varieties~\cite[Theorem 1.1]{Bir16b} we conclude that the projective varieties $E$ belong to a bounded family.
By~\cite[Theorem 3.34]{HK10} we know that the coefficients of $\Delta_E$ are standard, therefore the pairs $(E,\Delta_E)$ are log bounded by Lemma~\ref{from bounded to log bounded}.
Moreover, since $(E,\Delta_E)$ is $\epsilon$-log canonical then the coefficients of $\Delta_E$ are at most $1-\epsilon$, hence belong to a finite set of rational numbers.
By Lemma~\ref{from lb to slb}, conclude that the pairs $(E,\Delta_E)$ belong to a strictly log bounded family $\mathcal{P}$, which only depends on the dimension $d$ and the positive real number $\epsilon$.

By~\cite[Proposition 3.9]{Sho92}, we know that at codimension $2$ points of $Y$, every Weil divisor has Cartier index bounded by $p$, for some constant $p$ which only depends on $\epsilon$.
Hence, there exists a closed subset $Z$ on $Y$ of codimension at least $3$, so that the Cartier index of $E$ is a divisor of $p$ outside $Z$.
By Lemma~\ref{log bounded curve}, we may find an ample curve $C$ so that
\[
-(K_E+\Delta_E) \cdot C \leq M,
\]
for some constant $M$ which only depends on $\mathcal{P}$.
Thus, up to replacing $C$ with a rationally equivalent curve we may assume that $C$ does not intersect $Z$, so we have $p E \cdot C$ is a negative integer,
or equivalently, 
\[
- E \cdot C \in \mathbb{Z}_{>0}\left[ \frac{1}{p} \right].
\]
Moreover, since $(E,\Delta_E)$ belongs to a strictly log bounded family we conclude that the Cartier index of $-(K_E+\Delta_E)$ 
is bounded by a constant which only depends on $\mathcal{P}$~\cite[Lemma 2.25]{Bir16a}. In particular, 
\[
-(K_Y+E) \cdot C = 
-(K_E+\Delta_E)\cdot C 
\]
belongs to a finite set $\mathcal{F}$ of positive rational numbers,
which only depends on $\mathcal{P}$. 
Finally, observe that from the relation
\[
(K_Y+E  -(a_E(X,0)+1) E ) \cdot C = \pi^*(K_X)\cdot C =0,
\]
we conclude that 
\[
\frac{(a_E(X,0)+1)}{p} \leq (a_E(X,0)+1) (- E\cdot C) \in \mathcal{F},
\]
so
\[
a_E(X,0) \leq p \max\{\mathcal{F}\}-1,
\]
where the right hand side only depends on $\epsilon$ and $\mathcal{F}$.
Since $\mathcal{F}$ only depends on $\mathcal{P}$, and $\mathcal{P}$ only depends on $d$ and $\epsilon$,
we deduce that $l=p \max\{ \mathcal{F}\}-1$ only depends on $d$ and $\epsilon$. This proves the first statement.
Since $a_E(X,0)$ is a non-negative integer, we conclude that there are finitely many possible values for it,
proving the second statement.
\end{proof}

To conclude this subsection, we will prove that the mildest component of a klt singularity is indeed attained by a Koll\'ar component, or equivalently,
the infimum in Definition~\ref{mKc} is a minimum.

\begin{proposition}\label{mildest Kollar component is a minimum}
Let $(X,\Delta)$ be a klt singularity so that the coefficients of $\Delta$ are rational numbers and $x\in X$.
There exists a Koll\'ar component $\pi \colon Y \rightarrow X$ extracting a divisor $E\subset Y$ so that 
\[
{\rm \mathcal{M}\mathcal{C}}_x(X,\Delta) = a(Y,\Delta_Y+E).
\]
\end{proposition}

\begin{proof}
We proceed by contradiction.
Assume we have a sequence of Koll\'ar components $\pi_i \colon Y_i \rightarrow X$ over $x\in X$ for the pair $(X,\Delta)$
so that the total discrepancies of the pairs $(E_i, \Delta_{E_i})$ are in an infinite increasing sequence.
Since the coefficients of $\Delta_{Y_i}$ are fixed by ${\rm coeff}(\Delta)$, we conclude that the coefficients of
$\Delta_{E_i}$ belong to a finite set of rational numbers which only depend on ${\rm coeff}(\Delta)$~\cite[Lemma 5.3]{FM18}.
By~\cite[Theorem 1.1]{Bir16a}, Lemma~\ref{from bounded to log bounded}, and Lemma~\ref{from lb to slb} we conclude that the log pairs $(E_i,\Delta_{E_i})$ belong to a strictly log bounded family.
This contradicts the fact that the total log discrepancy can take only finitely many values on bounded families.
\end{proof}

\subsection{Finite morphisms}

In this subsection, we recall the index one cover of a log canonical singularity and the behaviour of log discrepancies under finite dominant morphisms.

\begin{definition}{\em
Let $(X,\Delta)$ be a pair and write $\Delta=\sum_i d_i \Delta_i$ where the $\Delta_i$'s are pairwise different prime divisors on $X$.
Given a quasi-finite morphism $\phi\colon X'\rightarrow X$ between normal varieties, we can write
\[
\phi^*(K_X+\Delta) = K_{X'} + \Delta',
\]
where
\[
\Delta':=\sum_i \sum_{f(E_j)=\Delta_i} (d_i(r_j+1)-r_j)E_j
\]
and $r_j$ is the ramification index at the generic point of $E_j$~\cite[2.1]{Sho92}.
The above formula is called the {\em pull-back formula} for quasi-finite morphisms.}
\end{definition}

The following lemma follows from the pull-back formula for quasi-finite morphisms.

\begin{lemma}\label{delta prime effective}
Let $(X,\Delta)$ be a  pair with standard coefficients one and
$\phi \colon X'\rightarrow X$ be a finite morphism of normal varieties.
Assume that for every prime divisor $E$ on $X'$ we have that $r_E+1$ divides $(1-d_i)^{-1}$,
where $r_E$ is the ramification index of $\phi$ at $E$ and $d_i$ is the coefficient of $\phi(E)$ at $\Delta$.
Then $\Delta'$ is an effective divisor whose coefficients are standard.
\end{lemma}

The following is a theorem of Zariski that is often used instead of resolution of singularities~\cite{Zar39}.

\begin{theorem}\label{birational extraction}
Let $Y'$ and $X$ be two integral schemes of finite type over a field over $\mathbb{Z}$,
and $f \colon Y'\rightarrow X$ a dominant morphism.
Let $D\subset Y'$ be a prime divisor and $\eta\in D$ the generic point.
Assume that $Y'$ is normal at $\eta$. We can define a sequence of schemes and rational maps as follows:
\begin{enumerate}
\item $X_0=X$ and $f_0=f$, 
\item If $f_i \colon Y' \dashrightarrow X_i$ is defined, then let $Z_i\subset X_i$ the closure of $f_i(\eta)$.
We define $X_{i+1}$ to be the blow-up of $X_i$ at $Z_i$ and $f_{i+1}\colon Y' \dashrightarrow X_{i+1}$ the induced rational map.
\end{enumerate}
For $j$ large enough $\dim(Z_j)\geq \dim(X)-1$ and $X_j$ is regular at the generic point of $Z_j$.
\end{theorem}

\begin{lemma}\label{finite log canonical}
With the assumptions of Lemma~\ref{delta prime effective}.
The following conditions hold:
\begin{enumerate}
\item For $\epsilon$ a non-negative real number the pair $(X',\Delta')$ is $\epsilon$-log canonical if and only if $(X,\Delta)$ is $\epsilon$-log canonical, and
\item if $(X,\Delta)$ is a log canonical pair with a unique log canonical center $x\in X$, $(X,\Delta)$ has a unique log canonical place, and $x'=\phi^{-1}(x)$ is a point, then $(X',\Delta')$ has a unique log canonical place.
\end{enumerate}
\end{lemma}

\begin{proof}
Let $\pi \colon Y \rightarrow X$ be a projective birational morphism and $Y'\rightarrow Y\times_{X} X'$ the normalization of the 
main component of the fiber product, then we have a commutative diagram
\begin{equation}\label{finite-birational}
 \xymatrix{
 Y' \ar[r]^-{\phi_Y}\ar[d]_-{\pi'}
  & Y \ar[d]^-{\pi}\\
 X' \ar[r]^-{ \phi } & X
 }
\end{equation}
and we can write 
\[
{\pi'}^*(K_{X'}+\Delta')=K_{Y'}+\Delta'_{Y'}+E' \text{ and }
\pi^*(K_X+\Delta)=K_Y+\Delta_Y+E. 
\]
Here $\Delta'_{Y'}$ is the strict transform of $\Delta'$ on $Y'$.
By Lemma~\ref{delta prime effective}, we know that $(X',\Delta')$ is a log pair.
Thus, we have the relation
\[
\phi_Y^*(K_Y+\Delta_Y+E) = K_{Y'}+\Delta'_{Y'}+E'.
\]
Therefore, by the pull-back formula for quasi-finite morphisms we have that 
\[
a_{E}(X',\Delta')+1 = r_E(a_E(X,\Delta)+1),
\]
where $r_E$ is the ramification index of $\phi_Y$ at the generic point of $E$.
By Theorem~\ref{birational extraction}, we know that for every divisorial valuation on $X'$
we can find $\pi \colon Y \rightarrow X$ so that the center of this valuation on $Y'$ is a divisor,
where $Y'$ is as in the commutative diagram~\eqref{finite-birational}.
This proves the first statement.

Now we turn to prove the second statement by contradiction.
By the proof of the first statement we know that the log canonical pair $(X',\Delta')$ has at least one log canonical place,
and all its log canonical places map onto $x'\in X'$. Assume that $(X',\Delta')$ has more than one log canonical place.
Let $\pi \colon Y'\rightarrow X'$ be a dlt modification of $(X',\Delta')$ and write 
\[
{\pi'}^*(K_{X'}+\Delta')= K_{Y'}+\Delta'_{Y'}.
\]
Applying~\cite[Theorem 5.48]{KM98} to a log resolution of $(Y',\Delta'_Y)$ we deduce that $\lfloor \Delta'_{Y'} \rfloor$ is connected. 
Moreover, since $(X',\Delta')$ has more than one log canonical place, the divisor 
$\lfloor \Delta'_{Y'}\rfloor$ is connected and has at least two irreducible components.
Since the intersection of two log canonical centers is a union of log canonical centers~\cite[Theorem 1.1]{Amb11}, we deduce that
$(Y',\Delta'_{Y'})$ has infinitely many log canonical places.
Thus, $(X',\Delta')$ has infinitely many log canonical places.
By Theorem~\ref{birational extraction}, we conclude that each of those log canonical places appears 
in a commutative diagram as in~\eqref{finite-birational}.
Moreover, at least one log canonical places of $(X',\Delta')$ is exceptional over the dlt model $(Y,\Delta_Y)$ of $(X,\Delta)$.
Therefore, $(X,\Delta)$ has at least two log canonical places.
This provides the needed contradiction and the claim follows.
\end{proof}

\begin{definition}\label{def: index one cover}{\em
Let $(X,\Delta)$ be a klt pair with standard coefficients and $x\in X$ a point.
Consider $a$ the smallest positive integer so that $a(K_X+\Delta)\sim 0$ on a neighborhod
of $x\in X$, or equivalently, $a$ is the Cartier index of $K_X+\Delta$ at $x\in X$.
Therefore we have an isomorphism $\mathcal{O}_X(a(K_X+\Delta))\simeq \mathcal{O}_X$, we can choose a 
nowhere zero section 
\[
s\in H^0(X,\mathcal{O}_X(a(K_X+\Delta))),
\]
and consider $\phi \colon X'\rightarrow X$ the corresponding cyclic cover.
The ramification index of $\phi$ at a prime divisor $E$ which maps onto $\Delta_i$ 
is exactly $d_i-1$, where $d_i={\rm coeff}_{\Delta_i}(\Delta)$.
Moreover, $\phi$ is ramified only at the support of $\Delta$.
Therefore, from the pull-back formula for quasi-finite morphisms we know that 
\[
K_{X'}= \phi^*(K_X+\Delta).
\]
We call $\phi$ the {\em index one cover} of the klt pair $(X,\Delta)$ locally at $x\in X$~\cite[Notation 4.1]{Fuj01}.
}
\end{definition}

\subsection{Complements}

In this subsection, we prove that klt singularities with plt blow-ups admits local complements with a unique log canonical place.

\begin{definition}{\em 
Let $(X,\Delta)$ be a log canonical pair and $X\rightarrow Z$ be a contraction of normal quasi-projective varieties.
We say that $\Gamma \geq 0$ is a {\em strong $(\delta,n)$-complement} over $z\in Z$ of $(X,\Delta)$ if the following conditions hold:
\begin{itemize}
\item $(X,\Delta+\Gamma)$ is an $\delta$-log canonical pair, and
\item $n(K_X+\Delta+\Gamma)\sim 0$ over a neighborhood of $z$.
\end{itemize}
In the case that $X\rightarrow Z$ is the identity, then we say that $\Gamma$ is a {\em local} strong $(\delta,n)$-complement around $x\in X$ for $(X,\Delta)$.
On the other hand, if $Z={\rm Spec}(k)$ for some field $k$, then we say that $\Gamma$ is a {\em global} strong $(\delta,n)$-complement for the pair $(X,\Delta)$.}
\end{definition}

\begin{remark}
In this notation, a pair $(X,\Delta)$ is exceptional at $x\in X$ if and only if every local complement at $x\in X$ has a unique log canonical place.
\end{remark}

\begin{lemma}\label{good complement}
Let $(X,\Delta)$ be a $d$-dimensional klt pair with standard coefficients.
Assume that $(X,\Delta)$ admits an $\epsilon$-plt blow-up at $x\in X$ extracting the exceptional divisor $E$.
There exists a natural number $n$, only depending on $d$ and $\epsilon$, 
and a boundary $\Gamma$ on $X$ so that the following conditions hold:
\begin{itemize}
\item $n(K_X+\Delta+\Gamma) \sim 0$ on a neighborhood of $x\in X$, and
\item $K_Y+\Delta_Y+\Gamma_Y +E =\pi^*(K_X+\Delta+\Gamma)$ is an $\epsilon$-plt pair on a neighborhood of $x\in X$.
\end{itemize}
Moreover, we may assume that the boundary divisors $\Gamma$ and $\Delta$ do not share prime components.
\end{lemma}

\begin{proof}
We will construct a strong $(0,n)$-complement for the divisor $-(K_Y+\Delta_Y+E)$ with respect to the morphism $\pi\colon Y \rightarrow X$ around $x\in X$.
This complement will push-forward to a $(0,n)$-complement for $(X,\Delta)$ locally around $x\in X$.
In order to do so, we will do adjunction to $E$, produce a global complement on $E$ and then pull-back to a neighborhood of $E$.

By adjunction, we can write
\[
(K_Y+\Delta_Y+E)|_E = K_E+\Delta_E
\]
is an $\epsilon$-log canonical pair and the coefficients of $\Delta_E$ belong to a set of rational numbers 
satisfying the descending chain condition with rational accumulation points (see, e.g.~\cite[Lemma 5.3]{FM18}).
Hence, by~\cite[Theorem 1.3.2]{FM18}, we can find $n$ only depending on $d$ and $\epsilon$, 
and a global strong $(\epsilon,n)$-complement $\Gamma_E$ for the pair $(E,\Delta_E)$.
From the proof of~\cite[Theorem 1.3]{FM18} we may assume that $\Gamma_E$ do not share prime components with $\Delta_E$.
Without loss of generality we may assume that $n\Delta_Y$ is a Weil divisor.

Let $\pi_Y \colon W \rightarrow Y$ be a log resolution of $(Y,\Delta_Y+E)$, and write
\[
-N_W := \pi_Y^*(K_Y+\Delta_Y+E) = K_W+\Delta_W +E_W,
\]
where $E_W$ is the strict transform of $E$ on $W$.
We define 
\[
L_W:= -nK_W - nE_W - \lfloor (n+1)\Delta_W \rfloor.
\]
Let $P_W$ be the unique integral effective divisor on $W$ so that 
\[
\Lambda_W := (n+1)\Delta_W - \lfloor (n+1)\Delta_W \rfloor + P
\]
is a boundary on $W$ so that $(W,\Lambda_W)$ is plt and $\lfloor \Lambda_W \rfloor=E_W$.
We claim that $P_W$ is an exceptional divisor over $Y$. 
Indeed, if $D$ is a prime divisor on $W$ which is not contracted on $Y$ then we have that
\[
{\rm coeff}_D(\lfloor (n+1)\Delta_W \rfloor) = {\rm coeff}_D( n\Delta_W),
\]
because $n\Delta_Y$ is integral.
Therefore 
\[
{\rm coeff}_D( (n+1)\Delta_W - \lfloor (n+1)\Delta_W \rfloor ) =
{\rm coeff}_D( \Delta_W )=
{\rm coeff}_{\pi_Y(D)}( \Delta_Y) \in (0,1).
\]
By definition, we have that 
\[
L_W+P_W-E_W =K_W+\Lambda_W-E_W + (n+1)N_W,
\]
is the sum of the klt pair $K_W+\Lambda_W-E_W$ and the nef and big divisor $(n+1)N_W$ over a neighborhood of $x\in X$.
Shrinking around $x\in X$ we may assume that $X$ is affine, then $(n+1)N_W$ is nef and big over $X$.
By the relative version of Kawamata-Viehweg theorem~\cite[Theorem 1-2-5]{KMM87}, we have a surjection
\begin{equation}\label{surjectivity}
H^0(L_W+P_W) \rightarrow H^0((L_W+P_W)|_{E_W}).
\end{equation}
We denote by $\Gamma_{E_W}$ the pull-back of $\Gamma_E$ to $E_W$.
Observe that we have 
\[
(L_W+P_W)|_{E_W} \sim G_{E_W} :=  n\Gamma_{E_W}+n\Delta_{E_W}- \lfloor (n+1)\Delta_{E_W} \rfloor +P_{E_W},
\]
where $P_{E_W}:=P_W|_{E_W}$ and $\Delta_{E_W}:=\Delta_W|_{E_W}$.
The divisor $G_{E_W}$ is integral and its coefficients are strictly greater than $-1$, therefore it is indeed effective.
By the surjectivity of~\eqref{surjectivity} there exists $0\leq G_W \sim L_W+P_W$ which restricts to $G_{E_W}$.
We denote by $G_Y$ the push-forward of $G_W$ to $Y$.
By pushing-forward the linear equivalence $L_W+P_W\sim G_W$ to $Y$, and using the fact that $P_W$ is $Y$-exceptional we get that
\[
0\leq G_Y \sim -n(K_Y+\Delta_Y+E).
\]
We define $\Gamma_Y := \frac{G_Y}{n}$.
Observe that by construction we have 
\[
n(K_Y+\Delta_Y+\Gamma_Y+E)\sim 0
\]
on a neighborhood of $x\in X$.
We claim that $\Gamma_Y|_{E}= \Gamma_E$.
Indeed, observe that we can define
\[
n \Gamma_W := G_W -P_W + \lfloor (n+1)\Delta_W \rfloor -n \Delta_W \sim nN_W \sim_{\qq,Y} 0,
\]
and $n\Gamma_W$ pushes-forward to $G_Y$ on $Y$, 
hence $\Gamma_W = \pi_Y^*(\Gamma_Y)$.
On the other hand, we have $n\Gamma_{E_W}= n \Gamma_W|_{E_W}$, 
which means that  $\Gamma_{E_W}=\Gamma_W|_{E_W}$. Thus, we have $\Gamma_Y|_E=\Gamma_E$ as claimed.

Finally, observe that 
\[
(K_Y+\Delta_Y+\Gamma_Y +E)|_E = K_W+\Delta_E+\Gamma_E,
\]
so by inversion of adjunction we conclude that $(Y,\Delta_Y+\Gamma_Y+E)$ is $\epsilon$-plt.
Moreover, since $\Gamma_E$ and $\Delta_E$ do not share prime components,
then $\Delta_Y$ and $\Gamma_Y$ do not share prime components as well.
Define $\Gamma =\pi_*(\Gamma_Y)$, and observe that 
\[
\pi^*(K_X+\Delta+\Gamma)= K_Y+\Delta_Y+\Gamma_Y+E.
\]
Thus, $(X,\Delta+\Gamma)$ is a log canonical pair with a unique log canonical place so that $n(K_X+\Delta+\Gamma) \sim 0$ on a neighborhood of $x\in X$.
Moreover, $\Delta$ and $\Gamma$ do not share prime components.
\end{proof}

\subsection{Examples}\label{sec: examples}

In this subsection, we give two examples to show that the $a$-log canonical 
and $\epsilon$-plt blow-up conditions of Theorem~\ref{finiteness e-plt blow-up} and Theorem~\ref{bounded cartier index e-plt} are indeed necessary.

\begin{example}{\em 
Let $X_n$ be the cone over a rational curve of degree $n$.
Blowing-up the vertex $\pi_n \colon Y_n \rightarrow X_n$ gives a log resolution so that
the pair $(Y_n,E_n)$ is log smooth. Hence, $\pi_n$ is a $1$-plt blow-up.
However, $a_{E_n}(X_n,0)= -1+\frac{2}{n}$, and the Cartier index of $X_n$ at the vertex depends on $n$.}
\end{example}

\begin{example}{\em
By~\cite[Proposition 5.1]{Hay99} we can construct terminal threefold singularities $X_m$ of index $m$ 
and extract two different divisors with discrepancies $1/m$ and $2/m$, respectively.
Hence, if there is any plt blow-up of $X_m$ it is an $\epsilon$-plt blow-up for some $\epsilon\leq 2/m$.}
\end{example}

\section{Proof of the main Theorem}

\begin{proof}[Proof of Theorem~\ref{finiteness e-plt blow-up}]
Let $(X,\Delta)$ be a log pair which admits an $\epsilon$-plt blow-up $\pi \colon Y \rightarrow X$ at $x\in X$.
By assumption $(X,\Delta)$ is $a$-lc at $x\in X$.
In particular, the minimal log discrepancy ${\rm mld}_x(X,\Delta)$ is strictly positive.

Let $\phi \colon X'\rightarrow X$ be the index one cover of the klt pair $(X,\Delta)$ locally at $x$
so that $\phi^*(K_X+\Delta)=K_{X'}$ (see~\cite[Notation 4.1]{Fuj01} or Definition~\ref{def: index one cover}).
We denote by $Y'$ the normalization of the main component of $X'\times_X Y$. 
Hence, we have a commutative diagram
\[
 \xymatrix{
 Y' \ar[r]^-{\phi_Y}\ar[d]_-{\pi'}
  & Y \ar[d]^-{\pi}\\
 X' \ar[r]^-{ \phi } & X
 }
\]
where $\pi'$ is birational and $\phi_Y$ is finite with the same degree as $\phi$.
We can write 
\[
\phi^*(K_X+\Delta)=K_{X'} \quad \text{ and } \quad
\phi_Y^*(K_Y+\Delta_Y+E)= K_{Y'}+E',
\]
where $E'$ is the reduced exceptional divisor contracted by $\pi'$.
We claim that the pair $(Y', E')$ is $\epsilon$-plt.
Indeed, by Lemma~\ref{good complement} we may find an effective divisor $\Gamma$ on $X$ so
that 
\[
\pi^*(K_X+\Delta+\Gamma)= K_Y+\Delta_Y+\Gamma_Y + E
\]
is an $\epsilon$-plt pair, where $\Gamma_Y$ is the strict transform of $\Gamma$ on $Y$.
Moreover, $(X,\Delta+\Gamma)$ has a unique log canonical place which corresponds to $E$,
and the divisors $\Delta$ and $\Gamma$ do not share prime components.
Therefore, by Lemma~\ref{delta prime effective} and Lemma~\ref{finite log canonical}, we conclude that 
\[
K_{X'}+\Gamma_{X'}= \phi^*(K_X+\Delta+\Gamma)
\]
is indeed a pair which has a unique log canonical place and its log discrepancies are either zero or greater than $\epsilon$.
By the commutativy of the diagram, we have that 
\[
{\pi'}^*(K_{X'}+\Gamma_{X'})= K_{Y'}+\Gamma_{Y'}+E'
\]
is an $\epsilon$-plt pair. Hence $(Y',E')$ is $\epsilon$-plt as well.
In particular, $\pi' \colon Y'\rightarrow X'$ is an $\epsilon$-plt blow-up at $x'\in X'$.

Observe that by construction the $\qq$-divisors $K_Y+\Delta_Y+E$ and 
$K_{Y'}+E'$ are both $\epsilon$-plt and anti-ample over $X$ and $X'$ respectively.
We define the following pairs by adjunction:
\[
K_{E'}+ \Delta_{E'} = (K_{Y'}+E')|_{E'}
\text{ and }
K_{E}+\Delta_{E} = ( K_Y+\Delta_Y+E)|_{E}.
\]
By the above considerations, we know that both pairs are anti-ample and $\epsilon$-lc.
By~\cite[Theorem 1.1]{Bir16b}, we know that the algebraic varieties $E$ and $E'$ belong to a bounded family which only depends on $d-1$ and $\epsilon$.
Moreover, the boundary divisors $\Delta_{E}$ and $\Delta_{E'}$ have coefficients that belong to a set with the descending chain condition.
By Lemma~\ref{from bounded to log bounded}, we conclude that the log pairs $(E,\Delta_{E})$ and $(E',\Delta_{E'})$ belong to a log bounded family which only depends on $d-1$, $\epsilon$,
and the derived set of standard coefficients.
Furthermore, by~\cite[Lemma 3.3]{Bir16a}, we know that the coefficients of $\Delta_{E}$ and $\Delta_{E'}$ belong to a set of
hyperstandard coefficients $\mathcal{H}(\mathcal{R})$ corresponding
to a finite set of rational numbers $\mathcal{R}$, which only depends on $\mathcal{S}$.
By the $\epsilon$-log canonical condition of the pairs $(E,\Delta_E)$ and $(E',\Delta_{E'})$
and the fact that the only accumulation point of $\mathcal{H}(\mathcal{R})$ is $1$, we conclude
that $\Delta_{E}$ and $\Delta_{E'}$ have coefficients in the finite set $\mathcal{H}(\mathcal{R})\cap [0,1-\epsilon)$ which only depends on $\epsilon$.
By Lemma~\ref{from lb to slb}, we deduce that the pairs $(E,\Delta_E)$ and $(E',\Delta_{E'})$ belong to a strictly log bounded family.
Denote by $\phi_E$ the restriction of $\phi_Y$ to $E'$ and observe that 
\[
\deg(\phi)=\deg(\phi_Y)=\deg(\phi_E)r_E,
\]
where $r_E$ denotes the ramification index of $\phi_Y$ at the generic point of $E'$.

Now we turn to prove that $\deg(\phi)$ has an upper bound which only depends on $d,a$ and $\epsilon$.
In order to do so, we just need to provide that there is an upper bound for $\phi_E$ and $r_E$.
Observe that we have
\[
\deg(\phi_E)= \frac{\vol\left(  -(K_{E'}+\Delta_{E'}) \right)}{\vol\left(  -(K_{E}+\Delta_{E}) \right)}.
\]
Therefore, by Lemma~\ref{bounded volume}, there is an upper bound for $\deg(\phi_E)$, which only depends on $d-1$ and $\epsilon$.
On the other hand, we have the relation 
\[
a_{E'}(X',0)+1 = r_E(a_E(X,\Delta)+1) \geq r_E a.
\]
By Lemma~\ref{bound of ld for Gorenstein} we know that $a_{E'}(X',0)+1$ has an upper bound which only depends on $d$ and $\epsilon$. We conclude that
$deg(\phi)$ has an upper bound which only depends on $d,a$ and $\epsilon$.

Thus, ${\rm mld}_x(X,\Delta)$ belongs to a discrete set which only depends on $d,a$ and $\epsilon$.
Finally, by Lemma~\ref{finite log canonical} we have that
\[
\deg(\phi) {\rm mld}_x(X,\Delta) \leq {\rm mld}_{x'}(X',0) \leq a_{E'}(X',0)+1,
\]
therefore ${\rm mld}_x(X,\Delta)$ belongs to a finite set which only depends on $d,a$ and $\epsilon$.
\end{proof}

\begin{proof}[Proof of Theorem~\ref{bounded cartier index e-plt}]
It follows from the bound on $\deg(\phi)$ given in the proof of Theorem~\ref{finiteness e-plt blow-up}.
\end{proof}

\begin{proof}[Proof of Corollary~\ref{acc and e-plt blow-ups}]
If there exists a sequence in $\mathcal{M}(d,\mathcal{S})_{0,\epsilon}$ which contradicts the ascending chain condition,
passing to a subsequence we may assume the sequence is strictly increasing.
Therefore, such infinite sequence belongs to $\mathcal{M}(d,\mathcal{S})_{a,\epsilon}$ for some positive real number $a$.
This contradicts Theorem~\ref{finiteness e-plt blow-up}.
\end{proof}

\begin{proof}[Proof of Corollary~\ref{acc exceptional singularities}]
The proof follows from Theorem~\ref{finiteness e-plt blow-up} and Lemma~\ref{exceptional singularities}.
\end{proof}

\end{document}